\documentclass[11pt,twoside]{article}

\usepackage{amsmath}
\usepackage{amsthm}
\usepackage{amssymb}
\usepackage{amscd}
\usepackage{setspace}
\usepackage{mathrsfs}
\usepackage{color}
\usepackage{multirow}
\usepackage{stmaryrd}	
\usepackage{mathdots}	
\usepackage{array}		
\usepackage{fancyhdr}
\input xy
\xyoption{all}
\usepackage{rotating}
\usepackage[unicode]{hyperref}	
\hypersetup{urlcolor=blue, colorlinks=true}		

\newenvironment{acknowledgements}{\list{}{\rightmargin0.75in\leftmargin0.75in}\item[]\textsc{Acknowledgements.}}{\endlist}

\newtheorem{theorem}{Theorem}[section]
\newtheorem{proposition}[theorem]{Proposition}

\newtheorem{lemma}[theorem]{Lemma}

\newtheorem{conjecture}[theorem]{Conjecture}

\theoremstyle{definition}

\newtheorem{remark}[theorem]{Remark}

\newcolumntype{x}[1]{>{\centering\arraybackslash$}m{#1}<{$}}

\newcommand{\mtrx}[4]{\left(\begin{array}{cc} #1 & #2 \\ #3 & #4 \\ \end{array}\right)}

\newcommand{\mf}[1]{\mathfrak{#1}}

\newcommand{\QQ}{\mathbf{Q}}
\newcommand{\ZZ}{\mathbf{Z}}

\newcommand{\CC}{\mathbf{C}}

\newcommand{\Zp}{\mathbf{Z}_p}

\newcommand{\wt}[1]{\widetilde{#1}}
\newcommand{\wh}[1]{\widehat{#1}}
\newcommand{\ol}[1]{\overline{#1}}

\renewcommand{\mod}{\operatorname{mod}}
\renewcommand{\Re}{\operatorname{Re}}

\newcommand{\mc}[1]{\mathcal{#1}}

\DeclareMathOperator{\gl}{GL}

\DeclareMathOperator{\Hom}{Hom}

\DeclareMathOperator{\tr}{tr}

\DeclareMathOperator{\Art}{Art}

\DeclareMathOperator{\ad}{ad}

\DeclareMathOperator{\Gal}{Gal}

\DeclareMathOperator{\Frob}{Frob}

\DeclareMathOperator{\Ind}{Ind}

\DeclareMathOperator{\ord}{ord}
\DeclareMathOperator{\Sym}{Sym}

\usepackage[titletoc,title]{appendix}

\usepackage{relsize}

\newcommand{\bigast}{\mathop{\mathlarger{\mathlarger{\mathlarger{\ast}}}}}
\renewcommand{\phi}{\varphi}
\newcommand{\bSel}{\overline{\operatorname{Sel}}}
\newcommand{\chic}{\chi_p}	
\renewcommand{\mod}[1]{\text{ (mod }#1)}

\DeclareMathOperator{\Std}{Std}
\DeclareMathOperator{\KL}{KL}
\DeclareMathOperator{\an}{an}
\DeclareMathOperator{\arith}{arith}
\DeclareMathOperator{\Gr}{Gr}
\DeclareMathOperator{\Gross}{Gross}
\DeclareMathOperator{\cts}{cts}
\newcommand{\rightisom}{\overset{{}_{\sim}}{\longrightarrow}}

\fancypagestyle{plain}{
\fancyhf{} 

}
\hypersetup{pdfauthor={Robert Harron},pdftitle={The exceptional zero conjecture for symmetric powers of CM modular forms: the ordinary case},pdfkeywords={\textit{p}-adic \textit{L}-functions, \textit{L}-invariants, Special values of \textit{L}-functions, Iwasawa theory, CM modular forms}}

\begin{document}
\title{The exceptional zero conjecture for symmetric powers of CM modular forms: the ordinary case}
\author{Robert Harron}
\date{\today\footnote{Original preprint version: September 24, 2011. This is a postprint version.\newline
		Published version: \textit{International Mathematics Research Notices} \textbf{2013}, no.~16, art.~ID rns161, pp.~3744--3770; \href{http://dx.doi.org/10.1093/imrn/rns161}{doi:~10.1093/imrn/rns161}\newline
		Department of Mathematics, University of Wisconsin--Madison\newline
		}}
\maketitle
\thispagestyle{plain}

\begin{abstract}
We prove the exceptional zero conjecture for the symmetric powers of CM cuspidal eigenforms at ordinary primes. In other words, we determine the trivial zeroes of the associated $p$-adic $L$-functions, compute the $L$-invariants, and show that they agree with Greenberg's $L$-invariants. In an appendix, we prove a functional equation for some of the $p$-adic $L$-functions we construct.
\end{abstract}

\maketitle

\tableofcontents

\section{Introduction}

	The exceptional zero conjecture concerns the zeroes of $p$-adic $L$-functions occurring at critical integers for a \textit{trivial reason} (namely, the vanishing of the interpolation factor). Given a ($p$-ordinary) motive $M$, there is conjecturally a $p$-adic function $L_p(s,M,\chi)$ (the \textit{analytic $p$-adic $L$-function of $M$}) interpolating the twisted values of the (usual) archimedean $L$-function $L(s,M,\chi)$ at its critical integers (in the sense of Deligne's article \cite{D79}). This interpolation takes the form
	\begin{equation}
		L_p(a,M,\chi)=\mc{E}(a,M,\chi)\frac{L(a,M,\chi^{-1})}{\Omega_{a,M,\chi}}
	\end{equation}
	where $\mc{E}(a,M,\chi)$ is a fairly simple arithmetic fudge factor called the \textit{interpolation factor}, $\Omega_{a,M,\chi}$ is some period (present to obtain algebraicity of values), and $a$ varies over the critical integers of $M$. For certain choices of $a$ and $\chi$ (and $M$), it can happen that $\mc{E}(a,M,\chi)$ vanishes, causing a trivial vanishing of the $p$-adic $L$-function. Such an occurrence is called a \textit{trivial zero} (or an \textit{exceptional zero}) of $L_p(s,M,\chi)$. Since the $p$-adic interpolation of $L$-values is a fruitful approach to understanding their arithmetic, it is worth attempting to correct the loss of information that occurs at a trivial zero. This is where the exceptional zero conjecture comes in. The factor $\mc{E}(a,M,\chi)$ is a product of ``Euler-like'' factors, some of which may vanish for a specific triple $(a_0,M,\chi_0)$. Denote by $e(a_0,M,\chi_0)$ the number of vanishing factors and let $\mc{E}^+(a_0,M,\chi_0)$ be $\mc{E}(a_0,M,\chi_0)$ with the vanishing factors removed. Supposing $L(a_0,M,\chi_0)\neq0$, the exceptional zero conjecture for $(a_0, M, \chi_0)$ can be then be stated in two parts.

	\begin{conjecture}[Exceptional zero conjecture, vaguely]\footnote{
	We say \textit{vaguely} here not only because we have left several terms undefined, but mostly since our notation $L_p(s,M,\chi)$ is more evocative than accurate. Really there are $p-1$ branches of the $p$-adic $L$-function and the limit should be taken along the branch where the trivial zero occurs. For the precise formulas we prove, see Theorems \ref{TheoremA} and \ref{TheoremB} below.
	}\mbox{}\label{conj:Linvar}
		\begin{enumerate}
			\item The order of vanishing of $L_p(s,M,\chi_0)$ at $s=a_0$ is $e(a_0,M,\chi_0)$. We may therefore define the ``analytic $L$-invariant'' $\mc{L}_p^{\an}(a_0,M,\chi_0)\in\CC_p^\times$ by
				\[ \lim_{s\rightarrow a_0}\frac{L_p(s,M,\chi_0)}{(s-a_0)^{e(a_0,M,\chi_0)}}=\mc{L}_p^{\an}(a_0,M,\chi_0)\mc{E}^+(a_0,M,\chi_0)\frac{L(a_0,M,\chi_0^{-1})}{\Omega_{a_0,M,\chi_0}}.
				\]
			\item
				There is an ``arithmetic $L$-invariant'' $\mc{L}_p^{\arith}(a_0,M,\chi_0)\in\CC_p$, defined in terms of the arithmetic of $M$, such that
				\[\mc{L}_p^{\an}(a_0,M,\chi_0)=\mc{L}_p^{\arith}(a_0,M,\chi_0).
				\]
		\end{enumerate}
	\end{conjecture}
	Part (i) conjectures a way of recuperating an interpolation property by looking at the order $e(a_0,M,\chi_0)$ derivative of $L_p(s,M,\chi)$, whereas part (ii) is a supplement to an \textit{Iwasawa--Greenberg Main Conjecture} in the setting of a trivial zero, namely it gives an arithmetic meaning to an analytically defined $L$-invariant.

	In Theorems \ref{TheoremA} and \ref{TheoremB} below, we prove this conjecture for motives arising from symmetric powers of $p$-ordinary newforms \textit{with complex multiplication}. More specifically, the adjoint motive of a newform, which has important applications to its deformation theory, is the symmetric square motive twisted by the inverse of the determinant and we consider appropriate determinant twists of higher symmetric powers. The arithmetic $L$-invariant we use is that suggested by Greenberg in \cite{G94}. We will denote it $\mc{L}_p^{\Gr}(a_0,M,\chi_0)$. In the CM case, the symmetric powers decompose into a sum of Tate twists of modular forms and sometimes a Dirichlet character, with the trivial zero coming from the latter factor. We prove our theorems by reducing them to these already known cases. In particular, the determination of the $L$-invariant is reduced to the results for Dirichlet characters of Ferrero--Greenberg in \cite{FeG78} and Gross and Koblitz in \cite{GrKo79} and \cite[\S4]{Gr81}. In the specific case of the symmetric square of a CM elliptic curve, part of our results were obtained by Dabrowski and Delbourgo in \cite[Theorem 3.6.1]{DabDel97}. Other related results in the literature on symmetric powers of modular forms mostly concern the determination of Greenberg's $L$-invariant. The case of a CM elliptic curve was done in \cite{G94}. Greenberg also covers the case of elliptic curves with split, multiplicative reduction at $p$ in that paper. Generalizations of this latter result to higher weights and to Hilbert modular forms have been obtained by Hida (see, for example, \cite{Hi07}). However, the most difficult case to consider is the non-CM case when $p$ does not divide the level of the modular form $f$. A formula for Greenberg's $L$-invariant was obtained by Hida in the symmetric square case in \cite{Hi04} under some technical assumptions (see also chapter 2 of the author's Ph.D. thesis \cite{H-PhD} for a slightly different approach). In \cite{PR98}, Perrin-Riou has a conditional result on the analytic $L$-invariant of the symmetric square of an elliptic curve. The author has addressed the next highest symmetric power $L$-invariant (i.e.\ the sixth power) in \cite{H10}. The author's current strategy for dealing with higher symmetric powers requires that $f$ not be CM, so it is important to deal with the (simpler) CM case separately.

	Based on what happens for his $L$-invariant of elliptic curves with CM or with split, multiplicative reduction at $p$, Greenberg has raised the question of whether or not the $L$-invariant is independent of the symmetric power $n$ that is taken (\cite[p.~170]{G94}). Our results show that independence holds for CM newforms.
	
	The appendix is devoted to proving Theorem \ref{thm:appendix}, which provides a functional equation for the $p$-adic $L$-functions we construct in the case where trivial zeroes occur.

	Finally, we wish to mention two ways to extend these results. The first is to consider the non-ordinary situation. In this case, the Dirichlet character occurring in the decomposition of the symmetric power does not contribute a trivial zero. In fact, the trivial zeroes come from the (odd weight) modular forms in the decomposition. This should lead to trivial zeroes of higher order and the determination of the $L$-invariant should reduce to the recent work of Denis Benois in \cite{Be11}. This is investigated in upcoming joint work with Antonio Lei (\cite{HLei12}). In another direction, one may pass from $\QQ$ to a totally real field $F$ and consider the $L$-invariants of symmetric powers of the Hilbert modular forms that arise from the induction of a Hecke character of a CM extension of $F$. Let us mention that the generalization of Gross' result has recently been studied by Dasgupta--Darmon--Pollack in \cite{DDP11}.

\begin{acknowledgements}
I would like to thank Rob Pollack, Jay Pottharst, Antonio Lei, and Piper Harron for some helpful conversations. I am also happy to thank the referee for suggestions that made the text clearer, for catching some typos, and for encouraging me to expand some of the content (including the appendix). The content of this article was worked out while I was a postdoctoral fellow at Boston University and a visiting assistant professor at the University of Wisconsin--Madison. I would like to thank these institutions for their hospitality.
\end{acknowledgements}

\subsection{Notation and conventions}
	Throughout, we fix a rational prime $p\geq3$, algebraic closures $\ol{\QQ}$ of $\QQ$ and $\ol{\QQ}_p$ of $\QQ_p$, and embeddings $\iota_\infty:\ol{\QQ}\rightarrow\CC$ and $\iota_p:\ol{\QQ}\rightarrow\ol{\QQ}_p$. The choice of $\iota_p$ determines an (arithmetic) Frobenius element $\Frob_p\in G_\QQ$ (up to inertia). Similarly, $\iota_\infty$ gives a complex conjugation $\Frob_\infty\in G_\QQ$. Denote by ${\chic:G_\QQ\rightarrow\ZZ_p^\times}$ the $p$-adic cyclotomic character. Let $\CC_p$ be the completion of $\ol{\QQ}_p$ and let ${\log_p:\CC_p^\times\rightarrow\CC_p}$ denote the Iwasawa branch of the $p$-adic logarithm (i.e.\ such that $\log_p(p)=0$).

\subsection{Dirichlet characters and \texorpdfstring{$p$-adic $L$-functions}{p-adic L-functions}}
	A Dirichlet character will be a character $\chi:(\ZZ/\mf{f}\ZZ)^\times\longrightarrow\CC^\times$. Its conductor will be denoted $\mf{f}_\chi$ and its Gauss sum $\tau(\chi)$. We will identify $\chi$ with a Galois character, also denoted $\chi$, using the normalization $\chi(\Frob_p)=\chi(p)$ for $p\nmid\mf{f}$. A Dirichlet character is even or odd according to whether $\chi(-1)=\chi(\Frob_\infty)$ is $+1$ or $-1$. The $L$-function of the Dirichlet character $\chi$ is defined for $\Re(s)>1$ by
	\[ L(s,\chi)=\sum_{n\geq1}\frac{\chi(n)}{n^s}.
	\]

	Let $\Gamma^\times:=\Gal(\QQ(\mu_{p^\infty})/\QQ)$ be the Galois group over $\QQ$ of the field of $p$-power roots of unity. We will consider our $p$-adic $L$-functions as coming from $p$-adic measures (or pseudo-measures) on $\Gamma^\times$ (see, for example, \cite{Cz00} for these notions). If $\mu$ is such a (pseudo)-measure, it can be evaluated at elements $\chi$ in (some subset of) $\mf{X}^\times:=\Hom_{\cts}(\Gamma^\times,\CC_p^\times)$, the continuous group homomorphisms valued in $\CC_p^\times$. The cyclotomic character gives an isomorphism $\chic:\Gamma^\times\rightisom\Zp^\times$ giving two natural characters $\omega,\langle\cdot\rangle\in\mf{X}^\times$ obtained by composing the cyclotomic character with the natural projections onto the first and second factors of $\Zp^\times=\mu_{p-1}\times(1+p\Zp)$. The finite order characters in $\mf{X}^\times$ are just those arising from $p$-power conductor Dirichlet characters (using $\iota_\infty$ and $\iota_p$). Each such character can be written uniquely as $\chi=\omega^{a}\phi$ where $0\leq a<p-1$ and $\phi$ is a character of $1+p\Zp$ (and hence necessarily even). If $M$ is some ``object'' and $\mu_M$ is a (pseudo-)measure attached to it, we will denote by $L_{p,i}(s,M)$ the $p$-adic function on a subset of $\ZZ_p$ given by
	\[ L_{p,i}(s,M):=\int_{\Gamma^\times}\omega^i\langle\cdot\rangle^sd\mu_M
	\]
	and call it the $i$th branch of the $p$-adic $L$-function of $M$. We use the shorthand
	\[ L_p(s,M,\chi):=\int_{\Gamma^\times}\chi\chic^sd\mu_M
	\]
	where $\chi$ is assumed to be of finite order and which makes sense for $s\in\ZZ$.

\subsection{Hecke characters and CM modular forms}
	We will be studying holomorphic cusp forms of integral weight $k\geq2$ on the upper-half plane. We refer to \cite{Ri77} for the notion of newform with Nebentypus and to sections 3 and 4 of that article for the notion of, and basic facts concerning, CM modular forms. In particular, the term newform will imply a normalized cuspidal eigenform for the Hecke algebra generated by all Hecke operators $T_n$. Our conventions regarding Hecke characters also follow \textit{loc.\ cit.} Finally, if
	\[ f(z)=\sum_{n\geq1}a_ne^{2\pi inz}
	\]
	is the $q$-expansion of the newform $f$, then its $L$-function is defined for $\Re(s)>(k+1)/2$ by
	\[ L(s,f)=\sum_{n\geq1}\frac{a_n}{n^s}.
	\]

\section{Symmetric powers of ordinary CM modular forms}\label{sec:symCM}
	Let $f$ be a $p$-ordinary\footnote{In the sense that $\ord_p(\iota_p\iota_\infty^{-1}a_p)=0$, where $a_p$ is the $p$th Fourier coefficient.} newform of weight $k\geq2$, of $\Gamma_1$-level prime to $p$ with Nebentypus character $\psi$. Let $\QQ(f)\subseteq\CC$ be the field generated by the Fourier coefficients of $f$. Let $K$ be the completion of $\iota_p\iota_\infty^{-1}\QQ(f)$ in $\ol{\QQ}_p$ and let $\rho_f:G_\QQ\longrightarrow\gl(V_f)$ be the contragredient of the $p$-adic Galois representation (occurring in \'etale cohomology) attached to $f$ by Deligne (\cite{D71}) on the two-dimensional vector space $V_f$ over $K$. Its determinant is $\psi\chic^{k-1}$. By the ordinarity assumption, (exactly) one of the roots of the Hecke polynomial $x^2-a_px+\psi(p)p^{k-1}$ of $f$ is a $p$-adic unit (in $\iota_p(\ol{\QQ})$); we denote it by $\alpha_p$ and denote the non-unit root by $\beta_p$. Wiles (\cite[Theorem~2.1.4]{W88}) shows that, for $f$ as above,
	\begin{equation}\label{eqn:rhoford}
		\rho_f|_{G_p}\sim\mtrx{\psi\chic^{k-1}\delta^{-1}}{\ast}{0}{\delta}
	\end{equation}
	where $G_p$ is the decomposition group at $p$ determined by the embedding $\iota_p$ and $\delta$ is the (non-trivial) unramified character sending $\Frob_p$ to $\alpha_p$.
	
	Now suppose $f$ has complex multiplication. Then, there is an imaginary quadratic field $F/\QQ$, an algebraic Hecke character $\wh{\eta}$ of $F$ of type $(k-1,0)$ and a (continuous) character $\eta:G_F\rightarrow \overline{\QQ}_p^\times$, corresponding to $\wh{\eta}$ under class field theory (and the choices of $\iota_\infty$ and $\iota_p$), such that $\rho_f=\Ind_F^\QQ\eta$. Letting $\tau$ denote an element in $G_\QQ$ not in $G_F$, we thus have
	\begin{equation}\label{eqn:rhofCM}
		\rho_f(g)=\left\{\begin{array}{ll}
						\mtrx{\eta(g)}{}{}{\eta(\tau g\tau^{-1})} & \text{if }g\in G_F, \\
						\\
						\mtrx{}{\eta(g\tau^{-1})}{\eta(\tau g)}{} & \text{if }g\not\in G_F.
					\end{array}
		\right.
	\end{equation}
	\begin{lemma}
		If $f$ is as above, i.e.\ a $p$-ordinary newform of weight $k\geq2$, $\Gamma_1$-level prime to $p$, Nebentypus character $\psi$, with CM by $F$, then the decomposition group at $p$, $G_p$, is contained in $G_F$. In particular, $p$ splits in $F$.
	\end{lemma}
	This is well-known, but we include a proof for the convenience of the reader. We also remark that if $p$ is inert, then $f$ is not ordinary at $p$.
	\begin{proof}
		For $\gamma\in I_p$, the inertia group at $p$, we have from \eqref{eqn:rhoford} that $\tr(\rho_f(\gamma))=\chic(\gamma)^{k-1}+1$. Since $\chic$ maps $I_p$ onto $\Zp^\times$, we may pick $\gamma_0\in I_p$ so that $\chic(\gamma_0)\neq\pm1$. Hence, $\tr(\rho_f(\gamma_0))\neq0$, so $\gamma_0\in G_F$, and furthermore $\rho_f(\gamma_0)$ has distinct eigenvalues. Denote by $W_1,W_2$ its (distinct) eigenspaces. These must provide the basis used in \eqref{eqn:rhofCM}. Suppose there is some $\gamma\in G_p$ such that $\gamma\not\in G_F$. Then, its matrix in a basis given by $W_1$ and $W_2$ is antidiagonal. It is a simple exercise in linear algebra to check that an antidiagonal matrix and a diagonal matrix with distinct eigenvalues cannot be simultaneously upper-triangularized. This contradicts \eqref{eqn:rhoford}.
	\end{proof}

	Let $\theta_F:G_\QQ\rightarrow\ZZ_p^\times$ be the quadratic character associated to the extension $F/\QQ$ by class field theory (so that $\theta_F(g)=-1$ if and only if $g\not\in G_F$). Then, $\theta_F(p)=1$ and since complex conjugation acts non-trivially on $F$, $\theta_F(\Frob_\infty)=-1$, i.e.\ $\theta_F$ is odd.

	For the rest of this section, we will consider even $n=2m$. This is a necessary condition for a trivial zero to occur; see the proof of Theorem \ref{TheoremA} for a discussion of the case of odd $n$.
	Taking the $n$th symmetric power of $\rho_f$ yields
	\begin{equation}
		\Sym^n\!\rho_f(g)=\begin{cases}
						\left(\begin{array}{cccc}
							\eta^n(g) \\
							&\eta^{n-1}(g)\eta(\tau g\tau^{-1}) \\
							&&\ddots \\
							&&&\eta^n(\tau g\tau^{-1})
						\end{array}\right) & \text{if }g\in G_F, \\
						\\
						\left(\begin{array}{cccc}
							&&&\eta^n(g\tau^{-1}) \\
							&&\iddots \\
							&\eta^{n-1}(\tau g)\eta(g\tau^{-1}) \\
							\eta^n(\tau g)
						\end{array}\right) & \text{if }g\not\in G_F
					\end{cases}
	\end{equation}
	in some basis, say, $v_0,\dots, v_n$. The entry in the middle of this matrix representation is given by
	\[	\begin{cases}
			\eta^m(g)\eta^m(\tau g\tau^{-1})	& \text{if }g\in G_F, \\
			\eta^m(\tau g)\eta^m(g\tau^{-1})	& \text{if }g\not\in G_F.
		\end{cases}
	\]
	We may rewrite this as $(\theta_F\det\rho_f)^m$. To identify the remaining entries, we change basis to $v_0,v_n,v_1,v_{n-1},v_2,\dots, v_{n/2}$. The matrix is then the direct sum of $(\theta_F\det\rho_F)^m$ and two-by-two blocks that we can write as
	\[	\begin{cases}
			\mtrx{\eta^{n-2i}(g)(\theta_F(g)\det\rho_f(g))^i}{}{}{\eta^{n-2i}(\tau g\tau^{-1})(\theta_F(g)\det\rho_f(g))^i}	& \text{if }g\in G_F, \\
			\\
			\mtrx{}{\eta^{n-2i}(g\tau^{-1})(\theta_F(g)\det\rho_f(g))^i}{\eta^{n-2i}(\tau g)(\theta_F(g)\det\rho_f(g))^i}{}	& \text{if }g\not\in G_F,
		\end{cases}
	\]
	for $0\leq i\leq m-1$. As alluded to in the introduction, we are interested in generalizing results obtained for the $L$-invariant of $\ad^0\!\rho_f$, the trace zero endomorphisms of $\rho_f$. In the representation theory of $\gl(2)$,
	\[ \ad^0\Std\cong(\Sym^2\!\Std)\otimes\operatorname{det}^{-1},
	\]
	where $\Std$ denotes the standard two-dimensional representation of $\gl(2)$. We thus view
	\begin{equation}
		\rho_n:=(\Sym^n\!\rho_f)\otimes(\det\rho_f)^{-m}
	\end{equation}
	as the proper generalization of the adjoint representation. We remark that the associated motive has weight $0$ so that the centre of the functional equation is $s=1/2$ and it will turn out that the trivial zeroes will occur at the near-central points. Twisting the above matrices by $(\det\rho_f)^{-m}$, we obtain
	\begin{equation}
		\rho_n\cong\theta_F^m\oplus\bigoplus_{i=0}^{m-1}\left(\left(\Ind_F^\QQ\eta^{2(m-i)}\right)\otimes\chic^{(m-i)(1-k)}\otimes\psi^{i-m}\theta_F^i\right).
	\end{equation}
	\begin{proposition}
		For each integer $j$ with $1\leq j\leq m$, there is a $p$-ordinary newform $f_j$ of weight $2j(k-1)+1$, level $N_j|N\mf{f}^2_{\psi^{-j}\theta_F^{1-j}}$, Nebentypus character $\theta_F$ with CM by $F$ such that
		\begin{equation}
			\rho_n\cong\theta_F^m\oplus\bigoplus_{j=1}^m\rho_{f_j}(-j(k-1)).\label{eqn:symndecomp}
		\end{equation}
	\end{proposition}
	\begin{proof}
	Since $\wh{\eta}$ is of type $(k-1,0)$, the Artin reciprocity map applied to $\eta^{2(m-i)}$ gives us Hecke characters of $F$ of type $(2(m-i)(k-1),0)$ that we denote $\wh{\eta}_j$ (for $1\leq j:=m-i\leq m$). These in turn yield newforms $f_j^\prime$ of weight $2j(k-1)+1$, which we must twist by $\psi^{-j}\theta_F^i$. This latter operation may result in imprimitive forms, but we let $f_j$ denote the unique newform (of level $N_j|N\mf{f}^2_{\psi^{-j}\theta_F^i}$) whose Hecke eigenvalues agree with those of the possibly imprimitive form $f^\prime_j\otimes\psi^{-j}\theta_F^i$ outside of the primes dividing $N\mf{f}_{\psi^{-j}\theta_F^i}$.
	
	The Nebentypus character of $f_j^\prime$ is determined by $\wh{\eta}_j$ (see \cite[p.~35]{Ri77}). Specifically, viewing $\wh{\eta}$ as a character on fractional ideals of $F$, we have that $\psi$ is the product of $\theta_F$ and the character $\wt{\eta}:a\mapsto\wh{\eta}((a))/a^{k-1}$. Similarly, the Nebentypus character $\psi_j^\prime$ of $f_j^\prime$ is $\wt{\eta}^{2j}\theta_F$. The effect on $\psi_j^\prime$ of twisting $f_j^\prime$ by $\psi^{-j}\theta_F^i$ is to multiply it by $(\psi^{-j}\theta_F^i)^2$. Thus, the Nebetypus character $\psi_j$ of $f_j$ is $\wt{\eta}^{2j}\theta_F(\wt{\eta}^{-j}\theta_F^{i-j})^2=\theta_F$ as claimed.
	
	By \eqref{eqn:rhofCM}, $\alpha_p$ is either $\eta(\Frob_p)$ or $\eta(\tau\Frob_p\tau^{-1})$. Either way, $\alpha_p^{2j}$ is a root of the Hecke polynomial of $f_j^\prime$ and therefore the $p$th Fourier coefficient of $f_j^\prime$ is a $p$-adic unit. Multiplying the latter by $\psi^{-j}(p)\theta_F^i(p)$ yields the $p$th Fourier coefficient of $f_j$, which shows that $f_j$ is $p$-ordinary.
	\end{proof}
	This proposition has two important consequences. First, we get a product decomposition of the $L$-function of $\rho_n$ as
	\begin{equation}\label{eqn:symnLdecomp}
		L(s,\rho_n)=L(s,\theta_F^m)\cdot\prod_{j=1}^mL(s+j(k-1),f_j).
	\end{equation}
	Second, $\rho_n$ is self-dual, i.e.\ $\rho_n\cong\rho_n^\vee$, where $\rho_n^\vee$ is the linear dual of $\rho_n$. This can be seen as follows. Recall that the contragredient $g^\vee$ of a newform $\displaystyle g=\sum_{n\geq1}b_nq^n$ of weight $\kappa$ is given by
	\[ g^\vee=\sum_{n\geq1}\ol{b}_nq^n
	\]
	and that $\rho_g^\vee\cong\rho_{g^\vee}(-\kappa)$. But \cite[Proposition 3.3]{Ri77} shows that the $f_j$ have totally real Fourier coefficients. Consequently, the representations $\rho_{f_j}(-j(k-1))$ are self-dual. We will not use the isomorphism $\rho_n\cong\rho_n^\vee$ in the main body of this paper, though it underlies the content of the appendix.
	
	The critical integers $C_{n,k}$ of $L(s,\rho_n)$ depend on the parities of both $m$ and $k$. When $m$ is odd and $k$ is even, the critical integers are
	\begin{equation}
		C_{n,k}=\left\{-(k-1)+1,-(k-1)+3,\dots,-2,0,1,3,5,\dots,k-3,k-1\right\}.
	\end{equation}
	If $m$ and $k$ are both odd, then
	\begin{equation}
		C_{n,k}=\left\{-(k-1),-(k-1)+2,\dots,-2,0,1,3,5,\dots,k-4,k-2\right\}.
	\end{equation}
	Otherwise, $C_{n,k}$ consists of the positive even integers $\leq k-1$ and the negative odd integers $\geq -k$. These are simple to work out from the Hodge structures of the corresponding motives.

	In the next section, we will $p$-adically interpolate the values of $L(s,\rho_n)$ at these integers by taking the product of the $p$-adic $L$-functions attached to each $L$-function on the right-hand side of \eqref{eqn:symnLdecomp}. This is a sensible approach since the critical integers of $\rho_n$ are contained in the intersection of the critical integers of the $L$-functions on the right-hand side of the decomposition. Indeed, the critical integers of a cusp form of weight $k^\prime$ are $\{1,2,\dots,k^\prime-1\}$, those of an odd Dirichlet character (such as $\theta_F^m=\theta_F$ for $m$ odd) are the positive odd and non-positive even integers, and those of an even Dirichlet character (such as $\theta_F^m=\mathbf{1}$ for $m$ even) are the positive even and negative odd integers. We remark that this method of defining $p$-adic $L$-functions for the symmetric powers of CM newforms already appears in the work of Dabrowski (\cite{Dab93}).

\section{The analytic \texorpdfstring{$L$}{L}-invariant}
	As is appropriate, we refer the reader to the section ``Notation and conventions'' for our notation and conventions on $p$-adic $L$-functions. We remark that we will only consider the $p$-adic $L$-function of $\rho_n$ corresponding to the ordinary refinement.

	Suppose that $f$ is a $p$-ordinary newform of weight $k\geq2$, level $\Gamma_1(N)$, with $p\nmid N$, and Neben\-typus character $\psi$. Recall that $\alpha_p$ (resp.\  $\beta_p$) denotes the unit root (resp.\ non-unit root) of the Hecke polynomial of $f$. Then, through the work of Manin, Amice--V\'elu, Vi\v{s}ik, and Shimura, there exist $\Omega_f^+,\Omega_f^-\in\CC^\times$ and a $\CC_p$-valued $p$-adic measure $\mu_f$ on $\Gamma^\times$ such that for any finite order character $\chi\in\mf{X}^\times$ and $1\leq a\leq k-1$, an integer,
	\begin{equation}\label{eqn:MFp-adicL}
		\iota_p^{-1}\int_{\Gamma^\times}\chi\chic^ad\mu_f=\iota_\infty^{-1}\!\left(\frac{L(a,f,\chi^{-1})}{\Omega_{a,f,\chi}}\right)\left(1-\frac{p^{a-1}\chi(p)}{\alpha_p}\right)\left(1-p^{-a}\chi^{-1}(p)\beta_p\right),
	\end{equation}
	where, writing $\mf{f}_{\chi}=p^r$,
	\[ \frac{1}{\Omega_{a,f,\chi}}=\frac{\Gamma(a)\mf{f}_{\chi^{-1}}^a}{\tau(\chi^{-1})(-2\pi i)^a\alpha_p^r}\cdot
								\begin{cases}
									\displaystyle\frac{1}{\Omega_f^+} & \text{if }(-1)^a\chi(-1)=1 \\
									\\
									\displaystyle\frac{1}{\Omega_f^-} & \text{if }(-1)^a\chi(-1)=-1.
								\end{cases}
	\]
	As discussed in \cite[\S II.15]{MTT86}, no trivial zeroes can occur for $f$ as above (namely, $p$-ordinary with $p\nmid N$).

	Consequently, the trivial zeroes of $\rho_n$ must come from the Dirichlet character $\theta_F^m$. This character is trivial unless $m$ is odd. Let us study the trivial case first. It follows from the classical results of Kubota--Leopoldt and Iwasawa that there is a $\CC_p$-valued $p$-adic pseudo-measure\footnote{
	Let us remark that we have normalized our choices so that
	\[ L_{p,i}(s,\mathbf{1})=\begin{cases}
						L_p^{\KL}(s,\omega^{1-i}) & \text{if }i\text{ is odd,} \\
						L_p^{\KL}(1-s,\omega^i) & \text{if }i\text{ is even,}
					\end{cases}
	\]
	where $L_p^{\KL}(s,\chi)$ is the classical Kubota--Leopoldt $p$-adic $L$-function attached to the even character $\chi$.
	} $\mu_\mathbf{1}$ on $\Gamma^\times$ such that
	for any finite order character $\chi\in\mf{X}^\times$ and for an integer $a$
	\begin{equation}\label{eqn:evenDirp-adicL}
		\iota_p^{-1}\int_{\Gamma^\times}\chi\chic^ad\mu_\mathbf{1}=\iota_\infty^{-1}\!\left(\frac{L(a,\chi^{-1})}{\Omega_{a,\mathbf{1},\chi}}\right)\cdot\begin{cases}
												\left(1-p^{-a}\chi^{-1}(p)\right) & \text{if }(-1)^a\chi(-1)=-1, a\leq0 \\
												\left(1-p^{a-1}\chi(p)\right) & \text{if }(-1)^a\chi(-1)=1, a\geq1,
											\end{cases}
	\end{equation}
	where
	\begin{equation}
		\frac{1}{\Omega_{a,\mathbf{1},\chi}}=\begin{cases}
									1 & \text{if }(-1)^a\chi(-1)=-1, a\leq0, \\
									\frac{2\Gamma(a)\mf{f}^a_{\chi^{-1}}}{\tau(\chi^{-1})(-2\pi i)^a} & \text{if }(-1)^a\chi(-1)=1, a\geq1.
								\end{cases}
	\end{equation}
	Thus, no trivial zero can occur for the trivial character $\mathbf{1}$. Indeed, any trivial zero of this function must occur at $a=0$ or $1$. If $a=0$, then, for criticality, $\chi$ must be odd, and in particular non-trivial, so that $\chi^{-1}(p)=0$. The same reasoning applies to $a=1$.

	Now, let $\theta$ be an odd Dirichlet character of conductor prime to $p$. Then, there is a $\CC_p$-valued $p$-adic measure\footnote{
	For odd $\theta$, our choice is such that
	\[ L_{p,i}(s,\theta)=\begin{cases}
						L_p^{\KL}(s,\theta\omega^{1-i}) & \text{if }i\text{ is even,} \\
						L_p^{\KL}(1-s,\theta^{-1}\omega^i) & \text{if }i\text{ is odd.}
					\end{cases}
	\]
	} $\mu_\theta$ on $\Gamma^\times$ such that
	for any finite order character $\chi\in\mf{X}^\times$ and for an integer $a$
	\begin{equation}\label{eqn:Dirp-adicL}
		\iota_p^{-1}\int_{\Gamma^\times}\chi\chic^ad\mu_\theta=\iota_\infty^{-1}\!\left(\frac{L(a,\theta\chi^{-1})}{\Omega_{a,\theta,\chi}}\right)\cdot\begin{cases}
												\left(1-p^{-a}(\theta\chi^{-1})(p)\right) & \text{if }(-1)^a\chi(-1)=1, a\leq0, \\
												\left(1-p^{a-1}(\theta^{-1}\chi)(p)\right) & \text{if }(-1)^a\chi(-1)=-1, a\geq1,
											\end{cases}
	\end{equation}
	where
	\begin{equation}\label{eqn:OmegaDirp-adicL}
		\frac{1}{\Omega_{a,\theta,\chi}}=\begin{cases}
									1 & \text{if }(-1)^a\chi(-1)=1, a\leq0, \\
									\frac{2\Gamma(a)\mf{f}^a_{\theta\chi^{-1}}}{\tau(\theta\chi^{-1})(-2\pi i)^a} & \text{if }(-1)^a\chi(-1)=-1, a\geq1.
								\end{cases}
	\end{equation}
	Here, the parities work out so that trivial zeroes can occur. Again, a trivial zero can only occur for $a=0$ or $1$, in which cases one must also have that $\chi$ is trivial and $\theta(p)=1$.

	Inspired by the decomposition of $L$-functions in \eqref{eqn:symnLdecomp}, we define a ($\CC_p$-valued) $p$-adic measure (or pseudo-measure if $m$ is even) $\mu_{\rho_n}$ on $\Gamma^\times$ as a convolution
	\begin{equation}\label{eqn:rhondefn}
		\mu_{\rho_n}:=\mu_{\theta_F^m}\ast\bigast_{j=1}^m\left(\chic^{j(k-1)}\mu_{f_j}\right).
	\end{equation}
	Theorem \ref{thm:appendix} of the appendix proves a functional equation for $\mu_{\rho_n}$ when $m$ is odd.

	As we have seen that no trivial zero can occur when $m$ is even, we restrict from now on to the interesting case where $m$ is odd. It then follows from \eqref{eqn:Dirp-adicL} and \eqref{eqn:MFp-adicL} that for any finite order character $\chi\in\mf{X}^\times$ and any $a\in C_{n,k}$
	\begin{align}
		\iota_p^{-1}\int_{\Gamma^\times}\chi\chic^ad\mu_{\rho_n}&=\iota_\infty^{-1}\!\left(\frac{L(a,\theta_F\chi^{-1})}{\Omega_{a,\theta_F,\chi}}\cdot
														\prod_{j=1}^m\frac{L(a+j(k-1),f_j,\chi^{-1})}{\Omega_{a+j(k-1),f_j,\chi}}\right) \nonumber \\
											&\times \prod_{j=1}^m\left(1-\frac{p^{a+j(k-1)-1}\chi(p)}{\alpha_{j,p}}\right)\left(1-p^{-a-j(k-1)}\chi^{-1}(p)\beta_{j,p}\right) \nonumber \\
											&\times \begin{cases}
												\left(1-p^{-a}(\theta_F\chi^{-1})(p)\right) & \text{if }(-1)^a\chi(-1)=1, a\leq0 \\
												\left(1-p^{a-1}(\theta_F\chi)(p)\right) & \text{if }(-1)^a\chi(-1)=-1, a\geq1,
											\end{cases} \nonumber \\
											=\iota_\infty^{-1}\!\left(\frac{L(a,\rho_n,\chi^{-1})}{\Omega_{a,\rho_n,\chi}}\right)
											&\times \prod_{j=1}^m\left(1-\frac{p^{a+j(k-1)-1}\chi(p)}{\alpha_{j,p}}\right)\left(1-p^{-a-j(k-1)}\chi^{-1}(p)\beta_{j,p}\right) \nonumber \\
											&\times \begin{cases}
												\left(1-p^{-a}(\theta_F\chi^{-1})(p)\right) & \text{if }(-1)^a\chi(-1)=1, a\leq0 \\
												\left(1-p^{a-1}(\theta_F\chi)(p)\right) & \text{if }(-1)^a\chi(-1)=-1, a\geq1,
											\end{cases}
	\label{eqn:p-adicL}
	\end{align}
	where $\alpha_{j,p}$ (resp.~$\beta_{j,p}$) is the unit root (resp.~non-unit root) of the Hecke polynomial of $f_j$ and we have set
	\[ \Omega_{a,\rho_n,\chi}:=\Omega_{a,\theta_F,\chi}\cdot\prod_{j=1}^m\Omega_{a+j(k-1),f_j,\chi}.
	\]
	The trivial zeroes of $\rho_n$ occur at $L_p(0,\rho_n,\mathbf{1})=L_{p,0}(0,\rho_n)$ and $L_p(1,\rho_n,\mathbf{1})=L_{p,1}(1,\rho_n)$, i.e.\ in the $0$th and $1$st branch, respectively. The definition \eqref{eqn:rhondefn} gives product decompositions
	\begin{equation}\label{eqn:padicLdecomp}
		L_{p,i}(s,\rho_n)=L_{p,i}(s,\theta_F)\cdot\prod_{j=1}^mL_{p,i}(s+j(k-1),f_j).
	\end{equation}
	For each $j$, $s=j(k-1)$ and $1+j(k-1)$ are the near-central points for $f_j$. The work of Jacquet--Shalika in \cite{JSh76} shows that the archimedean $L$-values at these points are non-zero. We may thus conclude that
	\[ L_{p,i}(j(k-1),f_j)\neq0\neq L_{p,i}(1+j(k-1),f_j).
	\]
	Therefore, taking the derivative of both sides in \eqref{eqn:padicLdecomp}, we obtain for $i=0,1$
	\begin{equation}\label{eqn:derivdecomp}
		L^\prime_{p,i}(i,\rho_n)=L^\prime_{p,i}(i,\theta_F)\cdot\prod_{j=1}^mL_{p,i}(i+j(k-1),f_j).
	\end{equation}
	This reduces the problem of computing the $L$-invariants of $\rho_n$ to those of $\theta_F$. These values were obtained by Ferrero and Greenberg in \cite{FeG78}, Gross and Koblitz in \cite{GrKo79}, and placed within a theoretical framework by Gross in \cite{Gr81}. For an odd Dirichlet character $\theta$, let $R_p(\theta)$ be the $p$-adic regulator Gross defines in equation (2.10) of his article. Let $h_F$ denote the class number of $F$ and factor $(p)=\mf{p}\ol{\mf{p}}$ into prime ideals of $F$, where $\mf{p}$ corresponds to the embedding $\iota_p$. Let $\ol{\pi}$ be a generator of the principal ideal $\ol{\mf{p}}^{h_F}$ (well-defined up to a root of unity). We may now state the first theorem of our paper.

	\begin{theorem}\label{TheoremA}
		Let $p$ be an odd prime and let $f$ be a $p$-ordinary newform of weight $k\geq2$, $\Gamma_1$-level prime to $p$, Nebentypus character $\psi$, and CM by $F$. Let $\rho_f$ denote the $p$-adic Galois representation attached to $f$. If $n$ is a positive integer and $m=\lfloor n/2\rfloor$, then the $p$-adic $L$-function of ${\rho_n=(\Sym^n\!\rho_f)\otimes(\det\rho_f)^{-m}}$ has a trivial zero if and only if $n=2m$ with $m$ odd. In such a case, the trivial zeroes of $\rho_n$ are of order 1 and occur at the near-central points $L_p(0,\rho_n,\mathbf{1})$ and $L_p(1,\rho_n,\mathbf{1})$. The analytic $L$-invariants are given by
		\begin{equation}\label{eqn:thmA1}
			\mc{L}^{\an}_p(1,\rho_n,\mathbf{1})=-\frac{2\log_p(\ol{\pi})}{h_F}
		\end{equation}
		and
		\begin{equation}\label{eqn:thmA2}
			\mc{L}^{\an}_p(0,\rho_n,\mathbf{1})=-\mc{L}^{\an}_p(1,\rho_n,\mathbf{1})
		\end{equation}
		and occur in the following interpolation formulas for $i=0$ and $1$
		\begin{equation}\label{eqn:statement}
			L_{p,i}^\prime(i,\rho_n)=\mc{L}^{\an}_p(i,\rho_n,\mathbf{1})\mc{E}^+(i,\rho_n,\mathbf{1})\frac{L(i,\rho_n)}{\Omega_{i,\rho_n,\mathbf{1}}}
		\end{equation}
		where
		\[ \mc{E}^+(i,\rho_n,\mathbf{1})=\prod_{j=1}^m\left(1-\frac{p^{i+j(k-1)-1}\chi(p)}{\alpha_{j,p}}\right)\left(1-p^{-i-j(k-1)}\chi^{-1}(p)\beta_{j,p}\right).
		\]
		We also have the identities
		\begin{equation}\label{eqn:thmA3}
			\mc{L}^{\an}_p(1,\rho_n,\mathbf{1})=\mc{L}^{\an}_p(1,\theta_F,\mathbf{1})=R_p(\theta_F)=\mc{L}_p^{\Gr}(1,\theta_F,\mathbf{1}).
		\end{equation}
	\end{theorem}

	\begin{remark}
		Before proving the theorem, we make a few remarks.
		\begin{enumerate}
			\item This theorem solves part (i) of conjecture \ref{conj:Linvar}. Part (ii) follows from the three equalities
			\begin{equation}\label{rem:eqn1}
				\mc{L}_p^{\Gr}(1,\theta_F,\mathbf{1})=\mc{L}_p^{\Gr}(1,\rho_n,\mathbf{1}),
			\end{equation}
			\begin{equation}\label{rem:eqn2}
				\mc{L}_p^{\Gr}(1,\theta_F,\mathbf{1})=\mc{L}_p^{\Gr}(1,\rho_n^\vee,\mathbf{1}),
			\end{equation}
			and
			\begin{equation}\label{rem:eqn3}
				\mc{L}_p^{\Gr}(0,\rho_n,\mathbf{1})=-\mc{L}_p^{\Gr}(1,\rho_n^\vee,\mathbf{1}).
			\end{equation}
			See \S\ref{sec:GrLinvar} for the proof of these equalities and an explanation of our notation for Greenberg's $L$-invariant.
			\item This shows that the $L$-invariant of $\rho_n$ only depends on the CM field $F$, in particular it is independent of $n$. Greenberg has raised the question of whether the $L$-invariant is independent of $n$ for all newforms $f$. In \cite[p.~170]{G94}, he shows that this is true for his $L$-invariant when either $E$ is an elliptic curve with complex multiplication or with split, multiplicative reduction at $p$. The question for general $f$ is still open.
			\item In Theorem \ref{TheoremB}, we will show that $\mc{L}_p^{\an}(1,\rho_n,\mathbf{1})=-2\log_p(\alpha_p)/(k-1)$, thus generalizing the result of Dabrowski and Delbourgo in \cite[Theorem 3.6.1]{DabDel97}.
			\item The $L$-invariant of $\theta_F$ can be written in terms of Morita's $p$-adic Gamma function, see \cite[Proposition 1]{FeG78}.
			\item Using the results of Gross--Koblitz in \cite{GrKo79} and Katz \cite{K81}, these $L$-invariants can be expressed in terms of Gauss sums.
		\end{enumerate}
	\end{remark}

	\begin{proof}
		We have already explained everything in this theorem except for the order of the trivial zeroes, the values of the $L$-invariants, and the case of odd $n$. We begin with $n$ even. That the analytic $L$-invariants of $\rho_n$ and $\theta_F$ are equal follows from \eqref{eqn:derivdecomp}. As indicated in footnotes above, we have
		\begin{equation}
			L_{p,1}(s,\theta_F)=L_p^{\KL}(1-s,\theta_F\omega)
		\end{equation}
		and
		\begin{equation}
			L_{p,0}(s,\theta_F)=L_p^{\KL}(s,\theta_F\omega).
		\end{equation}
		That these Kubota--Leopoldt $p$-adic $L$-functions vanish to order exactly 1 at $s=0$ is Proposition~2 of \cite{FeG78}.

		Note that
		\[ L_{p,0}(s,\theta_F)=L_{p,1}(1-s,\theta_F),
		\]
		so
		\[ L_{p,0}^\prime(0,\theta_F)=-L_{p,1}^\prime(1,\theta_F).
		\]
		Combining this with the functional equation for Dirichlet $L$-functions gives \eqref{eqn:thmA2} and reduces the proof of the case $i=0$ to that of $i=1$: indeed, if
		\[ L_{p,1}^\prime(1,\theta_F)=\mc{L}_p^{\an}(1,\theta_F,\mathbf{1})\frac{L(1,\theta_F)}{\Omega_{1,\theta_F,\mathbf{1}}},
		\]
		then the functional equation says that the fraction on the right is simply $L(0,\theta_F)$.
		
		The equality between Gross' $p$-adic regulator and Greenberg's $L$-invariant is Proposition 5 of \cite{G94}. $L_p^{\KL}(1-s,\theta_F\omega)$ is what Gross denotes (except for the superscript that we have added) by $L^{\Gross}_p(\omega\otimes\theta_F,1-s)$. His conjecture 2.12 is that
		\[ (L_p^{\Gross})^\prime(\omega\otimes\theta_F,0)=R_p(\theta_F)A(\theta_F),
		\]
		where we see from \cite[(2.16)]{Gr81} that $A(\theta_F)=-L(0,\theta_F)$. This conjecture is proved in \S4 of his paper for the situation in which we find ourselves. Using the functional equation as above yields
		\[ (L_p^{\Gross})^\prime(\omega\otimes\theta_F,0)=-R_p(\theta_F)\frac{L(1,\theta_F)}{\Omega_{1,\theta_F,\mathbf{1}}}.
		\]
		Since
		\[ L^\prime_{p,1}(1,\theta_F)=-(L_p^{\Gross})^\prime(\omega\otimes\theta_F,0)
		\]
		we have proved \eqref{eqn:thmA3}.

		In \cite[\S4]{FeG78} and \cite[(4.12)]{GrKo79}, one finds
		\[ (L_p^{\KL})^\prime(0,\theta_F\omega)=\frac{4}{w_F}\log_p(\ol{\pi}),
		\]
		where $w_F$ is the number of roots of unity in $F$. By the analytic class number formula,
		\[ L(0,\theta_F)=\frac{2h_F}{w_F},
		\]
		so
		\[ (L_p^{\KL})^\prime(0,\theta_F\omega)=\frac{2\log_p(\ol{\pi})}{h_F}L(0,\theta_F),
		\]
		from which the formula in \eqref{eqn:thmA1} follows.

		It remains to deal with the case of odd $n$. Briefly, the $L$-function decomposes into a product of $L$-functions of modular forms (i.e.\ no Dirichlet character shows up in the decomposition). All of the modular forms that show up will once again be ordinary at $p$ and have level prime to $p$. So, as mentioned above, none of them will have trivial zeroes.
	\end{proof}

\section{The arithmetic \texorpdfstring{$L$}{L}-invariant}\label{sec:GrLinvar}
	We will now briefly explain how the results in Theorem \ref{TheoremA} are those predicted by Greenberg's theory of trivial zeroes and we will use \cite[Theorem A]{H-PhD} to prove another formula for the $L$-invariant. As a reference for Greenberg's theory see \cite[\S1]{H10} or \cite{G94}. We will follow the notation of the former.

	We begin by remarking that Greenberg's theory predicts when a Galois representation $V$ has a trivial zero at $s=1$. So, the $L$-invariant denoted by $\mc{L}(V)$ in \cite{H10} is what we denote here by $\mc{L}_p^{\Gr}(1,V,\mathbf{1})$. Furthermore, saying that a trivial zero of $V$ should occur at $s=a_0$ is the same as saying that it should occur at $s=1$ for $V(1-a_0)$.

	The discussion in \cite[\S1.3]{H10} on the expected location of trivial zeroes of the symmetric powers of $f$ applies here as well.\footnote{
	There the Nebentypus was assumed to be trivial, whereas here we only assume it has conductor prime to $p$. However, the same arguments go through without modification.
	}
	Therefore, we can say that the trivial zeroes we found on the analytic side in the previous section are exactly those predicted by Greenberg's arithmetic theory and they occur with the expected order. It then remains to check that the $L$-invariants are as predicted. Let $\rho_n$ be defined as in the previous section, with $n=2m$ and $m$ odd. By definition, $\mc{L}_p^{\Gr}(0,\rho_n,\mathbf{1})=\mc{L}(\rho_n(1))$ which in turn is defined\footnote{This definition is implicit in \cite{G94}.} to be $-\mc{L}((\rho_n(1))^\vee(1))$. Since $(\rho_n(1))^\vee(1)=\rho_n^\vee$, this gives equation \eqref{rem:eqn3} mentioned above. There is a condition that must be satisfied in order for Greenberg's definition of the $L$-invariant of $V$ to make sense, namely that the balanced Selmer group $\bSel_\QQ(V)$ vanishes. As in Proposition 1.3 of \cite{H10}, we have that $\bSel_\QQ(\rho_n)=H^1_f(\QQ,\rho_n)$ where the latter is the likely more familiar Bloch--Kato Selmer group defined in \cite{BK90}. Similarly for $\rho_n^\vee$. Using some deep results of Rubin's on the main conjecture for imaginary quadratic fields, we obtain the following result.
	\begin{proposition}
		The balanced Selmer groups $\bSel_\QQ(\rho_n)$ and $\bSel_\QQ(\rho_n^\vee)$ vanish.
	\end{proposition}
	\begin{proof}
		The decomposition in \eqref{eqn:symndecomp} gives that
		\[ \bSel_\QQ(\rho_n)=H^1_f(\QQ,\rho_n)=H^1_f(\QQ,\theta_F)\oplus\bigoplus_{j=1}^mH^1_f(\QQ,\rho_{f_j}(j(k-1)).
		\]
		The vanishing of $H^1_f(\QQ,\theta_F)$ is a classical result that can be reduced to the finiteness of the class number. For $H^1_f(\QQ,\rho_{f_j}(j(k-1)))$, the vanishing is a deep result. Kato explains in \cite[\S15]{Ka04} how to deduce it from Rubin's results on the main conjecture for imaginary quadratic fields. Note that we are looking at Selmer groups of modular forms at non-central points. The Selmer groups at the central point lie deeper and are the subject of the Birch and Swinnerton-Dyer conjecture.

		For $\bSel_\QQ(\rho_n^\vee)$, the vanishing follows as above. Indeed, $\rho_n^\vee$ decomposes into the Dirichlet character $\theta_F$ and twists of newforms $f_j^\vee$ which are contragredient to the $f_j$.
	\end{proof}

	Next, we state a lemma whose proof we omit.
	\begin{lemma}
		If $V=\bigoplus_{j=0}^mV_j$ is a decomposition of $V$ as $G_\QQ$-modules, then
		\[ \mc{L}_p^{\Gr}(1,V,\mathbf{1})=\prod\mc{L}_p^{\Gr}(1,V_j,\mathbf{1})
		\]
		where the product is over only those $V_j$ that are expected to have a trivial zero.
	\end{lemma}
	Applying this to the decomposition of $\rho_n$ with $V_0=\theta_F$ and $V_j=\rho_{f_j}(j(1-k))$, for ${j = 1},\dots,m$, we get that
	\[ \mc{L}_p^{\Gr}(1,\rho_n,\mathbf{1})=\mc{L}_p^{\Gr}(1,\theta_F,\mathbf{1}),
	\]
	since it is easy to verify that, for $j\geq1$, $V_j$ are not expected to have trivial zeroes. This equality is the one mentioned in \eqref{rem:eqn1}. The same reasoning shows that
	\[ \mc{L}_p^{\Gr}(1,\rho_n^\vee,\mathbf{1})=\mc{L}_p^{\Gr}(1,\theta_F,\mathbf{1}),
	\]
	which is equation \eqref{rem:eqn2} above. Let us now state and prove the second theorem of our paper.
	
	\begin{theorem}\label{TheoremB}
		With the hypotheses of Theorem \ref{TheoremA}, we have, for $i=0$ or $1$,
		\begin{equation}
			\mc{L}_p^{\an}(i,\rho_n,\mathbf{1})=\mc{L}_p^{\Gr}(i,\rho_n,\mathbf{1})
		\end{equation}
		and
		\begin{equation}
			\mc{L}_p^{\Gr}(1,\rho_n,\mathbf{1})=-\frac{2\log_p(\alpha_p)}{k-1}.
		\end{equation}
	\end{theorem}
	\begin{proof}
		The equality of the analytic $L$-invariant with Greenberg's $L$-invariant follows from equations \eqref{rem:eqn1}, \eqref{rem:eqn2}, and \eqref{rem:eqn3} proved above.

		To prove the formula for Greenberg's $L$-invariant, we compute the $n=2$ case using the author's result in \cite[Theorem~A]{H-PhD}. Indeed, we have shown in Theorem \ref{TheoremA} that the $L$-invariant is independent of $n$, so we are free to pick an $n$. For $n=2$, Theorem A of \cite{H-PhD} states that in the current situation
		\[ \mc{L}_p^{\Gr}(1,\rho_2,\mathbf{1})=-\frac{2\alpha_p^\prime}{\alpha_p}
		\]
		and we will now explain the notation $\alpha_p^\prime$. We will require some results of Hida theory and we use \cite{Hi11} as a reference. Given a $p$-ordinary newform $f$ of level prime to $p$, Hida theory provides a $p$-adic analytic family $\mc{F}$ of $p$-adic modular forms deforming the ordinary $p$-stabilization $f_p$ of $f$, where
		\[ f_p(z):=f(z)-\beta_pf(pz).
		\]
		The $p$-stabilization has the effect that the $p$th Fourier coefficient of $f_p$ is $\alpha_p$.
		We may write $\mc{F}$ as a $q$-expansion
		\[ \mc{F}=\sum_{n\geq1}a_n(s)q^n
		\]
		where $a_n(s)$ is a $p$-adic analytic function of $s$ on some neighbourhood of $k$ and $a_n(k)$ is the $n$th Fourier coefficient of $f_p$. In particular, $a_p(k)=\alpha_p$. Then $\alpha_p^\prime$ is defined as
		\[ \alpha_p^\prime:=\left.\frac{da_p(s)}{ds}\right|_{s=k}.
		\]
		Alternatively, $\mc{F}$ provides a Galois deformation of $\rho_f$ and $a_p(s)$ is the $p$-adic analytic function that interpolates the unit roots of Frobenius at $p$ in this deformation.
		
		We are thus trying to show that the logarithmic derivative of $a_p(s)$ at $s=k$ is given by $\log_p(a_p(k))/(k-1)$. This follows since $a_p(s)$ is an exponential function (to the correct base). Indeed, the Hida family of $f$ interpolates the inductions of the Hecke characters $\eta^r$ with varying $r$ (see, for example, \cite[\S7]{Hi86}). Explicitly, it is shown in \cite[p.~1337]{Hi11} that
		\begin{equation}\label{eqn:aps}
			a_p(s)=\zeta\exp_p\left((s-1)\frac{\log_p(\ol{\pi})}{h_F\log_p(1+p)}\log_p(1+p)\right)=\zeta\exp_p\left((s-1)\frac{\log_p(\ol{\pi})}{h_F}\right)
		\end{equation}
		where $\zeta$ is a root of unity; indeed, this formula is written there as
		\[ \zeta t^{\log_p(\ol{\pi})/(h_F\log_p(1+p))}
		\]
		where\footnote{Here, $\mathbb{I}$ is the irreducible component of Hida's ordinary $p$-adic Hecke algebra corresponding to $\mc{F}$.} $t=1+T\in\mathbb{I}$. Moreover, a $p$-adic analytic function is obtained from this latter expression by sending $t$ to $(1+p)^{s-1}$. Since the $p$-adic logarithm of a root of unity is zero, taking the logarithm on both sides of \eqref{eqn:aps} yields
		\begin{equation}\label{eqn:logpaps}
			\log_p(a_p(s))=(s-1)\frac{\log_p(\ol{\pi})}{h_F},
		\end{equation}
		so that the logarithmic derivative of $a_p(s)$ at $s=k$ is $\log_p(\ol{\pi})/h_F$. Plugging $s=k$ into \eqref{eqn:logpaps} gives
		\[ \log_p(\alpha_p)=\log_p(a_p(k))=(k-1)\frac{\log_p(\ol{\pi})}{h_F}.
		\]
		Therefore,
		\[ \mc{L}_p^{\Gr}(1,\rho_n,\mathbf{1})=-\frac{2\alpha_p^\prime}{\alpha_p}=-\frac{2\log_p(\alpha_p)}{k-1},
		\]
		as desired.
	\end{proof}
	\begin{remark}\mbox{}
		\begin{enumerate}
			\item This formula for the $L$-invariant generalizes that in \cite[Theorem 3.6.1]{DabDel97}. There, Dabrowski and Delbourgo are considering the case where $n=2$ and $f$ corresponds to an elliptic curve over $\QQ$ with CM. We would like to thank Antonio Lei for bringing this result to our attention.
			\item If one could prove that, for at least one $n$, the analytic $L$-invariant equals Greenberg's $L$-invariant independently of the work of Ferrero--Greenberg and Gross--Koblitz, then the proof above would provide a new proof of their formulas for $(L_p^{\KL})^\prime(0,\theta_F\omega)$.
		\end{enumerate}
	\end{remark}
	
\begin{appendices}
\section{}
In this appendix, we will prove a functional equation for the $p$-adic measure $\mu_{\rho_n}$ when $n\equiv2\mod{4}$. This result is not used in the main body of this article, but, as the referee has pointed out, it explains the relationship between the $L$-invariants at $s=0$ and $s=1$ expressed in equation \eqref{eqn:thmA2} and is a useful addition to the literature. Our approach is simply to use the classical functional equations for $L(s,\theta_F)$ and $L(s,f_j)$ (and their twists) in order to compare the interpolation properties of two sides of a $p$-adic functional equation. We remark that this approach to functional equations for $p$-adic $L$-functions is present in Coates' article \cite[p.~170]{Co91} on his joint work with Perrin-Riou.

We begin by recalling the classical functional equations.

\begin{proposition}
	Let $\theta$ be a primitive Dirichlet character and let $\epsilon_\theta=0$ or $1$ depending on whether $\theta$ is even or odd. Then, for all positive integers $a\equiv\epsilon_\theta\mod{2}$,
	\begin{equation}\label{eqn:ClassFuncEqDir}
		L(a,\theta)=\frac{\tau(\theta)(-2\pi i)^a}{2\Gamma(a)\mf{f}_\theta^a}L(1-a,\theta^{-1}).
	\end{equation}
\end{proposition}
See, for instance, \cite[p.~12]{Iw-LpL}. Note that the fraction on the right-hand side is exactly the period $\Omega_{a,\theta,\mathbf{1}}$. This explains our choice of period in \eqref{eqn:OmegaDirp-adicL}.

We will need to recall a bit more to state the functional equation for newforms. We use \cite[\S\S 4.3 and 4.6]{Miy-MF} as a reference. Let $g$ be a newform of weight $\kappa$, level $\Gamma_1(\mathcal{N})$, with Nebentypus $\theta$. Denote by $W_\mc{N}$ the level $\mc{N}$ (and weight $\kappa$) Atkin--Lehner operator given explicitly by
\[ W_\mc{N}(g)(z)=\frac{1}{(\sqrt{\mc{N}}z)^\kappa}g(-1/\mc{N}z).
\]
If $g$ is a (normalized) newform, then $W_\mc{N}(g)$ may not be normalized, and in fact will never be when $k$ is odd. Indeed, \cite[(4.6.18)]{Miy-MF} shows that $W_\mc{N}(g)$ is a multiple of $g^\vee$ and a simple computation shows that $W_\mc{N}^2$ acts by $(-1)^\kappa$. We will now specialize to the case where $\kappa$ is odd and the Fourier coefficients of $g$ are totally real. As pointed out in \cite[p.~34, Remark 2]{Ri77}, this implies that $g$ has CM by some imaginary quadratic field $F$ and $\theta=\theta_F$. Clearly, $g^\vee=g$, so $g$ is an eigenvector of $W_\mc{N}$. Denote its eigenvalue by $w_g$ and note that it must equal $\pm i$.

\begin{proposition}
	Suppose $g$ is as above and $\chi$ is a primitive Dirichlet character of conductor prime to $\mc{N}$. Then, for all $s\in\CC$,
	\begin{equation}\label{eqn:ClassFuncEqNewform}
		\left(\frac{\mf{f}_\chi\sqrt{\mc{N}}}{2\pi}\right)^{\!s}\!\cdot\Gamma(s)L(s,g,\chi)=i^\kappa w_g\theta(\mf{f}_\chi)\chi(-\mc{N})\frac{\tau(\chi)}{\tau(\chi^{-1})}\left(\frac{\mf{f}_\chi\sqrt{\mc{N}}}{2\pi}\right)^{\!\kappa-s}\!\cdot\Gamma(\kappa-s)L(\kappa-s,g,\chi^{-1}).
	\end{equation}
\end{proposition}
See \cite[Theorem 4.3.12]{Miy-MF}.

In order to prove a nice functional equation, we will judiciously pick our periods $\Omega_g^\pm$. We take advantage of the fact that if $\kappa$ is odd, then, for $a\in\ZZ$, we have that $a$ and $\kappa-a$ have opposite parities. Given a choice of period $\Omega^-_g$, we choose
\begin{equation}\label{eqn:omegagplus}
	\Omega_g^+=\begin{cases}
					w_g\sqrt{\mc{N}}\Omega_g^-	&\kappa\equiv1\mod{4}\\
					\displaystyle\frac{w_g}{\sqrt{\mc{N}}}\Omega_g^-	&\kappa\equiv3\mod{4}.
				\end{cases}
\end{equation}
The next lemma shows that these choices are acceptable.

\begin{lemma}
	If $L(1,g)\in(-2\pi i)\Omega^-_g\QQ(g)$, then $L(\kappa-1,g)\in(-2\pi i)^{\kappa-1}\Omega^+_g\QQ(g)$.
\end{lemma}
\begin{proof}
	By the above functional equation,
	\[ \frac{L(\kappa-1,g)\sqrt{\mc{N}}^{\pm1}}{w_g(-2\pi i)^{\kappa-1}\Omega^-_g}=\frac{1}{(\kappa-2)!}\cdot\frac{L(1,g)}{(-2\pi i)\Omega^-_g}\cdot
		\begin{cases}
			\sqrt{\mc{N}}^{1-\kappa}	&\kappa\equiv1\mod{4}\\
			\sqrt{\mc{N}}^{3-\kappa}	&\kappa\equiv3\mod{4},
		\end{cases}
	\]
	which is indeed in $\QQ(g)$.
\end{proof}

We now proceed to prove functional equations for the $p$-adic measures $\mu_{\theta_F}$ and $\mu_{f_j}$ used in the definition of $\mu_{\rho_n}$. In what follows, $\phi$ always denotes either the trivial character or a (primitive) $p$-power conductor Dirichlet character of the second kind (i.e.\ with conductor divisible by $p^2$). In particular, $\phi$ is always \textit{even}. We will use the well-known fact that the $b$th branch of a $p$-adic measure on $\Gamma^\times$ is determined by its values at $\omega^b\phi\langle\cdot\rangle^a$ for a fixed integer $a$ and varying $\phi$ (where we are allowed to omit finitely many characters). We remind the reader that all $p$-power conductor Dirichlet characters are of the form $\omega^b\phi$. We also recall the involution $\mu\mapsto\mu^\#$ defined by
\[ \int_{\Gamma^\times}\chi d\mu^\#=\int_{\Gamma^\times}\chi^{-1} d\mu
\]
for all (not necessarily finite order) characters $\chi\in\mf{X}^\times$.

\begin{proposition}
	Let $\theta$ be an odd quadratic Dirichlet character of conductor prime to $p$. Then,
	\begin{equation}\label{eqn:funceqnp-adicDirChar}
		d\mu_\theta=\chic d\mu_\theta^\#.
	\end{equation}
\end{proposition}
\begin{proof}
	For an even branch $b$, we evaluate both sides at $\omega^b\phi\neq\mathbf{1}$. Using the interpolation property \eqref{eqn:Dirp-adicL} and the classical functional equation, the left-hand side is
	\begin{align*}
		\int_{\Gamma^\times}\omega^b\phi d\mu_\theta&=\frac{L(0,\theta\omega^{-b}\phi^{-1})}{\Omega_{0,\theta,\omega^b\phi}} \\
			&=L(1,\theta\omega^b\phi)\frac{2\mf{f}_{\theta\omega^b\phi}}{\tau(\theta\omega^b\phi)(-2\pi i)} \\
			&=\frac{L(1,\theta\omega^b\phi)}{\Omega_{1,\theta,\omega^{-b}\phi^{-1}}}\\
			&=\int_{\Gamma^\times}\omega^{-b}\phi^{-1}\chic d\mu_\theta
	\end{align*}
	This is exactly the right-hand side of \eqref{eqn:funceqnp-adicDirChar} integrated against $\omega^b\phi$.
	
	For an odd branch $b$, we evaluate at $\omega^b\phi\langle\cdot\rangle$. Taking $b^\prime=1-b$, the above computation verifies the equation.
\end{proof}

Given a positive prime-to-$p$ integer $M$, let $\Art(M)$ denote its Artin symbol in $\Gamma^\times$, i.e. for all $\chi\in\mf{X}^\times$, we have $\chi(\Art(M))=\chi(M)$. The associated Dirac delta distribution (which is a $p$-adic measure) is $\delta_{\Art(M)}$ and is such that
\[\int_{\Gamma^\times}\chi d\delta_{\Art(M)}=\chi(M).
\]
\begin{proposition}
	Let $g$ be a $p$-ordinary newform of odd weight $\kappa$, level $\Gamma_1(\mc{N})$ with $p\nmid\mc{N}$, Nebentypus $\theta$, and totally real Fourier coefficients. Then,
	\begin{equation}
		\chic^{(\kappa-1)/2}d\mu_g=\mc{N}^{\epsilon_b}\cdot\left(\chic\left(\chic^{(\kappa-1)/2}d\mu_g\right)^\#\right)\ast d\delta_{\Art(\mc{N})}^\#,
	\end{equation}
	where $\epsilon_b\in\{0,1\}$ is such that $b\equiv\epsilon_b\mod{2}$.
\end{proposition}
\begin{proof}
	For an even branch $b$, we again evaluate both sides at $\omega^b\phi\neq\mathbf{1}$. Note that $\theta(\mf{f}_{\omega^b\phi})=1$ (since $\theta(p)=1$) and $(\omega^{b}\phi)(-1)=1$. Write $\mf{f}_{\omega^b\phi}=p^r$. Using the interpolation property \eqref{eqn:MFp-adicL} and the classical functional equation, the left-hand side is
	\begin{align}
		\int_{\Gamma^\times}\omega^b\phi\chic^{(\kappa-1)/2}d\mu_g&=\frac{L\left(\frac{\kappa-1}{2},g,\omega^{-b}\phi^{-1}\right)}{\Omega_{(\kappa-1)/2,g,\omega^b\phi}}\nonumber\\
		&=(\omega^{-b}\phi^{-1})(\mc{N})i^\kappa w_g\frac{\tau(\omega^{-b}\phi^{-1})}{\tau(\omega^b\phi)}\frac{\mf{f}_{\omega^b\phi}\sqrt{\mc{N}}}{2\pi}\frac{\Gamma\left(\frac{\kappa+1}{2}\right)}{\Gamma\left(\frac{\kappa-1}{2}\right)}L\left(\frac{\kappa+1}{2},g,\omega^b\phi\right)\nonumber\\
		&\text{\hspace{1.3em}}\times\frac{\Gamma\left(\frac{\kappa-1}{2}\right)\mf{f}_{\omega^b\phi}^{(\kappa-1)/2}}{\tau(\omega^{-b}\phi^{-1})(-2\pi i)^{(\kappa-1)/2}\alpha_p^r\Omega_g^{(-1)^{(\kappa-1)/2}}}\nonumber\\
		&=(\omega^{-b}\phi^{-1})(\mc{N})i^{\kappa-1}w_g\sqrt{\mc{N}}\frac{\Gamma\left(\frac{\kappa+1}{2}\right)\mf{f}_{\omega^b\phi}^{(\kappa+1)/2}}{\tau(\omega^b\phi)(-2\pi i)^{(\kappa+1)/2}\alpha_p^r\Omega_g^{(-1)^{(\kappa-1)/2}}}L\left(\frac{\kappa+1}{2},g,\omega^b\phi\right).\label{eqn:longmess}
	\end{align}
	If $\kappa\equiv1\mod{4}$, then
	\[ \frac{i^{\kappa-1}w_g\sqrt{\mc{N}}}{\Omega_g^{(-1)^{(\kappa-1)/2}}} = \frac{w_g\sqrt{\mc{N}}}{\Omega_g^+}=\frac{1}{\Omega_g^-}
	\]
	by our choice of periods in \eqref{eqn:omegagplus}. Similarly, when $\kappa\equiv3\mod{4}$,
	\[ \frac{i^{\kappa-1}w_g\sqrt{\mc{N}}}{\Omega_g^{(-1)^{(\kappa-1)/2}}} = \frac{-w_g\sqrt{\mc{N}}}{\Omega_g^-}=\frac{1}{\Omega_g^+}
	\]
	(since $w_g^{-1}=-w_g$). In both cases, this can be written as
	\[ \frac{1}{\Omega_g^{(-1)^{(\kappa+1)/2}}}.
	\]
	Substituting this back in to \eqref{eqn:longmess} yields
	\[ (\omega^{-b}\phi^{-1})(\mc{N})\frac{L\left(\frac{\kappa+1}{2},g,\omega^b\phi\right)}{\Omega_{(\kappa+1)/2,g,\omega^{-b}\phi^{-1}}},
	\]
	which is the desired result.

	For an odd branch $b$, we evaluate at $\omega^b\phi\langle\cdot\rangle$. Using that we already know the result on even branches, we get
	\begin{align*}
		\int_{\Gamma^\times}\omega^b\phi\langle\cdot\rangle\chic^{(\kappa-1)/2}d\mu_g&=\int_{\Gamma^\times}\omega^{b-1}\phi\chic^{(\kappa+1)/2}d\mu_g\\
			&=(\omega^{1-b}\phi^{-1})(\mc{N})\int_{\Gamma^\times}\omega^{1-b}\phi^{-1}\chic^{(k-1)/2}d\mu_g\\
			&=\mc{N}\langle\mc{N}\rangle^{-1}(\omega^{-b}\phi^{-1})(\mc{N})\int_{\Gamma^\times}\omega^{-b}\phi^{-1}\langle\cdot\rangle^{-1}\chic^{(\kappa+1)/2}d\mu_g,
	\end{align*}
	as desired.
\end{proof}

	We can now package up all these functional equations to obtain one for $\mu_{\rho_n}$. We denote by $\mf{f}_{\rho_n}$ the Artin conductor of $\rho_n$ and note that 
	\[ \mf{f}_{\rho_n}=\mf{f}_\theta\cdot\prod_{j=1}^mN_j.
	\]
	\begin{theorem}\label{thm:appendix}
		With the hypotheses of \ref{TheoremA} and $n=2m$, with $m$ odd, we have that
		\begin{equation}
			d\mu_{\rho_n}=\left(\frac{\mf{f}_{\rho_n}}{|\Delta_F|}\right)^{\epsilon_b}\cdot\left(\chic d\mu_{\rho_n}^\#\right)\ast d\delta_{\Art(\mf{f}_{\rho_n})}^\#\ast d\delta_{\Art(|\Delta_F|^{-1})}^\#,
		\end{equation}
		where $\epsilon_b\in\{0,1\}$ is such that $b\equiv\epsilon_b\mod{2}$.
	\end{theorem}
	\begin{proof}
		This basically follows from the definition of $\mu_{\rho_n}$. We just want to point out that $|\Delta_F|=\mf{f}_{\theta_F}$ and that its inverse occurs in the functional equation because $\mf{f}_\theta$ does not appear in \eqref{eqn:funceqnp-adicDirChar}.
	\end{proof}

	We end by pointing out that this implies that
	\[ L_{p,1}(s,\rho_n)=L_{p,0}(1-s,\rho_n)\left\langle\frac{\mf{f}_{\rho_n}}{|\Delta_F|}\right\rangle^{1-s}.
	\]
\end{appendices}


\end{document}